\documentclass[11pt,reqno]{amsart}
\usepackage[margin=1in]{geometry}
\usepackage[T1]{fontenc}
\usepackage{lmodern}
\usepackage{microtype}
\usepackage{amsmath,amssymb,amsthm,url}
\usepackage[utf8]{inputenc}
\usepackage[english]{babel}
\usepackage{enumerate}

\usepackage{graphicx}
\usepackage{mathrsfs}
\usepackage[colorlinks=true, pdfstartview=FitH, linkcolor=blue, citecolor=blue, urlcolor=blue]{hyperref}
\usepackage{mathtools}
\usepackage{needspace}
\newtheorem{theorem}{Theorem}[section]
\newtheorem{lemma}[theorem]{Lemma}
\newtheorem{proposition}[theorem]{Proposition}
\newtheorem{corollary}[theorem]{Corollary}

\newtheorem{conjecture}[theorem]{Conjecture}

\theoremstyle{definition}

\newtheorem{remark}[theorem]{Remark}
\allowdisplaybreaks
\newcommand{\F}{\mathbb F}
\newcommand{\Z}{\mathbb Z}
\newcommand{\E}{\mathbb E}
\newcommand{\supp}{\operatorname{supp}}
\newcommand{\Prob}{\mathbb P}

\title{A sharp inverse theorem for the quadratic large sieve}
\author{Ernie Croot}
\address{School of Mathematics\\ Georgia Institute of Technology\\ GA\\ United States}
\email{ernest.croot@math.gatech.edu}
\author{Junzhe Mao}
\address{School of Mathematics\\ Georgia Institute of Technology\\ GA\\ United States}
\email{jmao87@gatech.edu}
\author{Cosmin Pohoata}
\address{Department of Mathematics, Emory University, GA, United States}
\email{cosmin.pohoata@emory.edu}
\author{Chi Hoi Yip}
\address{Department of Mathematics, Hong Kong University of Science and Technology, Clear Water Bay, Hong Kong}
\email{machyip@ust.hk}
\subjclass[2020]{11N35, 11B30, 11P70}
\keywords{Inverse sieve conjecture, quadratic residues, inverse Goldbach problem}

\date{}

\begin{document}
\begin{abstract}
We prove that if $A\subseteq[N]$, $|A|\gg\sqrt N$, and $|A_p|\le p/2+O(1)$ for every prime $p\ll\log N$, then $A$ contains at least $\exp(c\sqrt{\log N}/\log\log N)$ elements in the image of a single integral quadratic $m\pm x^2$. A matching construction shows that this is optimal even when the quadratic may be chosen arbitrarily in $\Z[x]$. 

\end{abstract}

\maketitle
\section{Introduction}
For a set $A\subseteq\Z$ and a prime $p$, let $A_p=\{a\bmod p:a\in A\}$ denote the set of residue classes modulo $p$ occupied by $A$. For a positive integer $N$, write $[N]=\{1,\ldots,N\}$. For $A\subseteq[N]$, the large sieve \cite{Montgomery1978} and the larger sieve \cite{Gallagher1971} give upper bounds for $|A|$ in terms of the sizes of the sets $A_p$; see
\cite{FriedlanderIwaniec} for background. The inverse sieve problem asks whether sets nearly attaining these bounds necessarily exhibit algebraic structure. This question has been studied in, among other works, \cite{CrootLev2007,HelfgottVenkatesh2009,
Walsh2012,Walsh2014,GreenHarper2014,Shao2015,Shao2016,Hanson2020,CrootMaoYip2025,CrootMaoPohoataShefferYip2026}.

We focus on the square-root threshold. The condition $|A_p|\le p/2+O(1)$ for every prime implies $|A|\ll\sqrt N$ by either sieve. The set of squares shows that this bound is sharp up to a constant factor. More generally, the image of a fixed rational quadratic occupies about half the residue classes modulo $p$ for all but finitely many primes $p$, with the exceptions depending on its coefficients and denominator. The following is a weak form of the quadratic inverse sieve conjecture; see \cite[Conjecture~1.1]{Hanson2020} and \cite[Problem~47]{Green}. A related earlier formulation appears in \cite[Section~7.4]{CrootLev2007}.

\begin{conjecture}\label{conj1}
For every $\alpha>0$ there is a $\beta=\beta(\alpha)>0$ such that, for all sufficiently large $N$, the following holds. If $A\subseteq[N]$ satisfies $|A|\ge\alpha\sqrt N$ and $|A_p|\le(p+1)/2$ for every prime $p\le\sqrt N$, then there is a polynomial $q\in\mathbb Q[x]$ of degree two with
\[
        |A\cap q(\Z)|\ge\beta|A|.
\]
\end{conjecture}

Hanson \cite[Corollary~1.3]{Hanson2020} proved an unconditional result at the square-root threshold: under $|A|\gg\sqrt N$ and $|A_p|\le p/2+O(1)$ for every prime, some integral quadratic $m-x^2$ contains $\gg\log N$ elements of $A$. Our main theorem gives a much larger intersection while imposing the residue restriction only at primes up to $\frac14\log N$.

\begin{theorem}\label{thm:main}
Fix $\alpha>0$ and $C\ge0$. There are constants $c=c(\alpha,C)>0$ and $N_0=N_0(\alpha,C)$ such that the following holds. Let $N\ge N_0$, and let $A\subseteq[N]$ satisfy $|A|\ge\alpha\sqrt N$ and
\[
        |A_p|\le\frac p2+C         \qquad\text{for every prime }p\le\frac14\log N.
\]
Then there is an integral quadratic $q(x)=m\pm x^2$ such that
\[
        |A\cap q(\Z)|
        \ge\exp\left(c\frac{\sqrt{\log N}}{\log\log N}\right).
\]
\end{theorem}

The factor $1/4$ in the cutoff is inessential: the same lower bound
holds with $p\le\theta\log N$ for any fixed $\theta>0$, with constants
depending also on $\theta$. We also remark that a similar argument yields a nontrivial
intersection bound under the much weaker condition
$p-|A_p|\gg\sqrt p$ for primes $p\ll \log N$;
see Corollary~\ref{cor:square-root-deficit}.

We next give a construction showing that this exponential scale remains optimal even when one maximizes the intersection over all polynomials in $\Z[x]$ of degree two.

\begin{theorem}\label{thm:sharpness}
Fix $0<\theta<1/2$. There are constants $K=K(\theta)>0$ and $N_1=N_1(\theta)$ such that, for every $N\ge N_1$, there is a set $A\subseteq[N]$ with
\[
        |A|=\lfloor\sqrt N\rfloor,\qquad |A_p|\le\frac p2 \quad\text{for every prime }p\le\theta\log N,
\]
and
\[
        \max\{|A\cap q(\Z)|: q\in\Z[x], \deg q=2\}\le\exp\left(K\frac{\sqrt{\log N}}{\log\log N}\right).
\]
\end{theorem}

The maximum is taken over all integral quadratics, without any restriction on their coefficients, which may depend on $N$. Taking $\theta=1/4$ gives a matching upper bound under exactly the prime range in Theorem~\ref{thm:main}. We prove Theorem~\ref{thm:sharpness} in Section~\ref{sec:upper}.

Green and Harper \cite[Theorem~1.6]{GreenHarper2014} proved a three-alternative result for an infinite set $\mathcal A$ satisfying $|\mathcal A_p|\le(p+1)/2$ for every sufficiently large prime. In addition to eventual containment in a rational quadratic image or small counting functions along a sequence, the third alternative allows an intersection of size at most $X^{1/4+o_\psi(1)}$ with each fixed rational quadratic
image $\psi(\mathbb Q)$ inside $[X]$. They conjectured that this last alternative can be excluded.

\begin{conjecture}[Green--Harper]\label{conj:GH17}
Let $\mathcal A$ be a set of positive integers such that $|\mathcal A_p|\le(p+1)/2$ for all sufficiently large primes $p$. Then one of the following alternatives holds: 
\begin{enumerate}
\item there is a rational quadratic $\psi$ such that all but finitely many
      elements of $\mathcal A$ lie in $\psi(\mathbb Q)$;
\item for every integer $k$, there are arbitrarily large values of $X$ for
      which $|\mathcal A\cap[X]|<\sqrt{X}/(\log X)^k.$
\end{enumerate}
\end{conjecture}

Theorem~\ref{thm:main} has the correct exponential scale at its logarithmic cutoff, by Theorem~\ref{thm:sharpness}. Conjecture~\ref{conj:GH17}, however, assumes such a restriction for all sufficiently large primes, so it is natural to ask how the conclusion improves as the available prime range increases.

The next theorem gives stronger bounds as the range of primes increases. At the logarithmic cutoff it gives a density-dependent form of Theorem~\ref{thm:main}; extending the cutoff to any fixed power $(\log N)^{1+\epsilon}$ yields the stronger exponential scale
$\exp(c\sqrt{\log N/\log\log N})$.

\begin{theorem}\label{thm:largerrange}
Fix $\theta>0$. There are constants $c_\theta>0$ and $N_0=N_0(\theta)$ such that the following holds for every $N\ge N_0$ and $0<\delta<1$. Let $\theta\log N\le X\le(\log N)^2$. Suppose that $A\subseteq[N]$ is nonempty and
\[
 |A_p|\le(1-\delta)p
\]
for every prime $p\equiv3\pmod4$ with $\frac\theta2\log N<p\le X$. Then some integral quadratic $q(x)=m\pm x^2$ satisfies
\[
 |A\cap q(\Z)|
 \gg_\theta\frac{|A|}{\sqrt N}
 \exp\left(
 c_\theta\frac{\sqrt{\delta\log N}}{\log\log N}
 \sqrt{1+\log\frac{X}{\theta\log N}}
 \right).
\]
All constants are uniform in $\delta$. In particular, for any fixed $0<\epsilon\le1$, taking $X=(\log N)^{1+\epsilon}$ gives an exponential factor of the form 
$\exp\bigl(c_{\theta,\epsilon} \sqrt{\delta\log N/\log\log N}\bigr)$.
\end{theorem}

When $X=\theta\log N$, this becomes
\[
 |A\cap q(\Z)|\gg_\theta\frac{|A|}{\sqrt N}
 \exp\left(c_\theta\frac{\sqrt{\delta\log N}}{\log\log N}\right).
\]
Theorem~\ref{thm:main} is also recovered by taking $\theta=1/4$, $X=\theta\log N$, and $\delta=1/3$: for sufficiently large $N$, the hypothesis $|A_p|\le p/2+C$ implies $|A_p|\le(2/3)p$ throughout the required prime interval, while $|A|/\sqrt N\ge\alpha$ is absorbed by decreasing the constant in the exponent.

More generally, the factor $|A|/\sqrt N$ in Theorem~\ref{thm:largerrange} allows the conclusion of Theorem~\ref{thm:main}, with a smaller constant in the exponent, even when
\[
        |A|\ge\sqrt N\exp\left(-\kappa
        \frac{\sqrt{\log N}}{\log\log N}\right),
\]
provided $\kappa=\kappa(C)>0$ is sufficiently small and the same residue restriction holds. 

At the logarithmic cutoff, quadratic-character correlations and logarithmic averaging already suffice for Theorem~\ref{thm:main}; the latter has a relative-entropy interpretation (Remark~\ref{rem:relative-entropy}). To prove the stronger Theorem~\ref{thm:largerrange}, we use a local dichotomy: either logarithmic averaging yields a sufficient gain, or bounded local functions are amplified using the square of a truncated divisor sum of Selberg-sieve type. We outline the argument in Section~\ref{sec:outline}.

Green and Harper also obtained a partial robustness result for the quadratic inverse large sieve problem \cite[Theorem~1.4]{GreenHarper2014}. A rational quadratic has height at most $H$ if it can be written as $(ax^2+bx+c)/d$, where $a,b,c,d\in\Z$, $ad\ne0$, and $\max\{|a|,|b|,|c|,|d|\}\le H$. Roughly speaking, under $|A_p|+|B_p|\le p+1,$ large intersections with prescribed rational quadratics of controlled height force $A$ and $B$ to be mostly contained in those quadratics. This motivates the following conjecture.

\begin{conjecture}[Green--Harper]\label{conj:GH15}
Let $X_0\in\mathbb N$ and $\varrho>0$. Let $X$ be sufficiently large in terms of $X_0$ and $\varrho$. Suppose that $A,B\subseteq[X]$ and
\[
        |A_p|+|B_p|\le p+1
\]
for all primes $p\in[X_0,X^{1/4}]$. Then there is a constant $c=c(\varrho)>0$ such that one of the following alternatives holds:
\begin{enumerate}
\item either $|A\cap[X^{1/2}]|$ or $|B\cap[X^{1/2}]|$ is at most $X^{1/4-c}$;
\item there are rational quadratics $\psi_A,\psi_B$ of height at most $X^\varrho$ such that
      \[
        |A\setminus\psi_A(\mathbb Q)|\le X^{1/2-c},\qquad
        |B\setminus\psi_B(\mathbb Q)|\le X^{1/2-c}.
      \]
\end{enumerate}
\end{conjecture}

A further motivation is Ostmann's inverse Goldbach problem \cite{ostmann}. Let $\mathscr P$ denote the set of all primes, and write $S\sim T$ if the symmetric difference of two sets is finite.

\begin{conjecture}[Inverse Goldbach problem]\label{conj:inverse-goldbach}
There do not exist sets $\mathcal A,\mathcal B$ of positive integers, each with at least two elements, such that \[\mathscr P\sim\mathcal A+\mathcal B.\]
\end{conjecture}

This conjecture has a long history; see Hornfeck \cite{Hornfeck1954}, Elsholtz \cite{Elsholtz2001}, Elsholtz--Harper \cite{ElsholtzHarper2015}, and Shao \cite{Shao2016}. Green and Harper \cite[Theorem~1.8]{GreenHarper2014} proved that
Conjecture~\ref{conj:GH15} implies Conjecture~\ref{conj:inverse-goldbach}.
The underlying reason is that two large quadratic patterns cannot have a sumset consisting only of primes, apart from finitely many exceptions. A recent preprint by OpenAI~\cite{OpenAI2026InverseGoldbach} claims to prove Conjecture~\ref{conj:inverse-goldbach}, but it has not been peer reviewed.

We next establish a two-set result motivated by Conjecture~\ref{conj:GH15}. Our conclusion guarantees large intersections with integral quadratic images, rather than the near-containment in rational quadratic images predicted by that conjecture.

\begin{theorem}\label{thm:twoset-density}
There is an absolute constant $c>0$ such that, for every $C\ge0$, there is $N_0=N_0(C)$ for which the following holds. Let $N\ge N_0$, let $A,B\subseteq[N]$ be nonempty, and set
\[
 t=\log\max\left\{1,\frac{\sqrt N}{\min(|A|,|B|)}\right\}.
\]
Suppose that $t\le\sqrt{\log N}$ and
\begin{equation}\label{eq:pair}
        |A_p|+|B_p|\le p+C
        \qquad\bigl(p\le\min(|A|,|B|)\bigr).
\end{equation}
Then, for each $T\in\{A,B\}$, there is an integral quadratic $q_T(x)=m_T+\varepsilon_T x^2$, with $\varepsilon_T\in\{\pm1\}$, such that
\[
 |T\cap q_T(\Z)|
 \gg_C\frac{|T|}{\sqrt N}
 \exp\left(c\sqrt{\frac{\log N}{\log\log N+t}}\right).
\]
\end{theorem}

In particular, there are absolute constants $\kappa,c_1>0$ such that, for every fixed $C\ge0$ and all sufficiently large $N$ (depending on $C$), the condition~\eqref{eq:pair} together with
\[
 \min(|A|,|B|)\ge\sqrt N
 \exp\bigl(-\kappa(\log N)^{1/3}\bigr)
\]
implies that, for each $T\in\{A,B\}$, some integral quadratic $q_T(x)=m_T\pm x^2$ satisfies
\[
 |T\cap q_T(\Z)|\ge\exp\bigl(c_1(\log N)^{1/3}\bigr).
\]
Indeed, this follows from Theorem~\ref{thm:twoset-density} by taking $t\le\kappa(\log N)^{1/3}$ and choosing $\kappa>0$ sufficiently small. When $\min(|A|,|B|)\ge\sqrt N/(\log N)^D$ for fixed $D\ge0$, the theorem retains the stronger exponent $c_D \sqrt{\log N/\log\log N}$. Combining this theorem with the lower bounds of Elsholtz and Harper \cite{ElsholtzHarper2015} for the counting functions of hypothetical summands of the primes gives the following consequence.

\begin{corollary}\label{cor:goldbach-quadratic}
Suppose that $\mathscr P\sim\mathcal A+\mathcal B$, where $\mathcal A,\mathcal B\subseteq\mathbb N$ each contain at least two elements. Then there is a constant $c>0$ such that, for all sufficiently large $N$, there are integral quadratics
\[
        q_{\mathcal A}(x)=m_{\mathcal A}\pm x^2,
        \qquad
        q_{\mathcal B}(x)=m_{\mathcal B}\pm x^2
\]
for which
\[
        |\mathcal A\cap[N]\cap q_{\mathcal A}(\Z)|,\quad
        |\mathcal B\cap[N]\cap q_{\mathcal B}(\Z)|
        \ge\exp\left(c\sqrt{\frac{\log N}{\log\log N}}\right).
\]
\end{corollary}

\medskip
\noindent\textbf{Notation.}
All logarithms are natural. The letter $p$ always denotes a prime, and $\F_p$ is the field with $p$ elements. We write $X\ll Y$, or $Y\gg X$, when $|X|\le KY$ for an absolute constant $K$; subscripts record allowed dependencies. For a probability measure $\mu$ on a finite set, write $\E_\mu f=\sum_x\mu(x)f(x)$. For a nonempty finite set $A$, write $\E_A f=|A|^{-1}\sum_{a\in A}f(a)$, the uniform average over $A$.
A subscript such as $\E_{x\bmod p}$ indicates uniform averaging over that variable; unsubscripted expectations use the distribution specified by the context. For nonempty $A\subseteq[N]$, its residue distribution modulo $p$ is
\[
        \nu_{A,p}(y)=\frac{\#\{a\in A:a\equiv y\pmod p\}}{|A|}.
\]

\medskip
\noindent\textbf{Organization.}
Section~\ref{sec:outline} outlines the proof. Section~\ref{sec:local} proves Theorem~\ref{thm:main} by the logarithmic amplification. Section~\ref{sec:product} develops the local dichotomy and the complementary  truncated divisor-sum amplification. Section~\ref{sec:oneset} combines the two to prove Theorem~\ref{thm:largerrange}. Section~\ref{sec:twoset} proves Theorem~\ref{thm:twoset-density} and Corollary~\ref{cor:goldbach-quadratic} regarding two sets $A$ and $B$. Section~\ref{sec:upper} proves Theorem~\ref{thm:sharpness} on the matching construction.

\section{Outline of the proof}\label{sec:outline}
Let $\chi_p$ denote the quadratic character of $\F_p$.
For $p\equiv3\pmod4$, a sign $\varepsilon\in\{\pm1\}$,
and a shift $s\in\F_p$, the function
$1+\varepsilon\chi_p(y-s)$ counts the solutions of
$y=s+\varepsilon x^2$ in $\F_p$.
By the Chinese remainder theorem, products of these
local counts, with a common sign, count solutions
to a single quadratic congruence.

Let $\nu_p$ be the residue distribution of $A$ modulo
$p$. The quadratic-character correlations
\[
 \xi_p(s)=\sum_{y\in\F_p}\nu_p(y)\chi_p(y-s)
\]
satisfy
\[
 \sum_{s\in\F_p}\xi_p(s)^2
 =p\sum_{y\in\F_p}\nu_p(y)^2-1
 \ge\frac{p}{|A_p|}-1.
\]
Thus the residue restrictions force nontrivial
correlations. In Section~\ref{sec:local},  we use these correlations to
construct, for each prime $p$, a sign
$\varepsilon_p\in\{\pm1\}$ and a probability measure
$\lambda_p$ on $\F_p$ such that
\[
 Z_p(y)=\sum_{s\in\F_p}
 \lambda_p(s)(1+\varepsilon_p\chi_p(y-s))
\]
has a positive logarithmic average $\E_{\nu_p}\log Z_p$. After selecting primes with a common sign,
Jensen's inequality combines these local gains.
The Chinese remainder theorem and a lifting argument
then give a large intersection with a single integral
quadratic, proving Theorem~\ref{thm:main}.

For larger prime ranges, the product of these primes becomes too
large for this lifting argument. Section~\ref{sec:product} develops a
complementary method: a local dichotomy either gives a sufficient
logarithmic gain or supplies bounded functions with controlled moments. In the second case, we use the square of a truncated divisor sum of  Selberg-sieve type. Section~\ref{sec:oneset} combines the two alternatives and optimizes the parameters. Finally, Gallagher's larger sieve gives the additional input for the two-set theorem in Section~\ref{sec:twoset}.

\section{Logarithmic amplification}\label{sec:local}
\subsection{A local logarithmic gain}
Throughout this subsection, write $\E$ for uniform averaging over $\F_p$. We first recall the quadratic character-correlation identity: for
$y,z\in\F_p$,
\begin{equation}\label{eq:char-correlation}
        \sum_{s\in\F_p}\chi_p(y-s)\chi_p(z-s)
        =
        \begin{cases}
        p-1,&y=z,\\
        -1,&y\ne z.
        \end{cases}
\end{equation}
This identity is standard; see for example
\cite[Theorem 5.48]{LN97}.

We next establish a logarithmic-gain estimate for shifts with positive
quadratic-character correlations. It will also be used in
Section~\ref{sec:product}.

\begin{lemma}
\label{lem:positive-correlation}
Let $p$ be an odd prime and let $\nu$ be a probability measure on
$\F_p$. Write
\[
 \xi(s)=\sum_{y\in \F_p}\nu(y)\chi_p(y-s),\qquad
 W=\sum_{s\in \F_p}\xi(s)^2.
\]
\begin{enumerate}[(i)]
\item We have
\begin{equation}\label{eq:correlation-moment}
 W=p\sum_{y\in\F_p}\nu(y)^2-1.
\end{equation}
Moreover, for every real function $F$ on $\F_p$,
\begin{equation}\label{eq:change-measure}
 \left|\E_\nu F-\E F\right|\le \sqrt{W\,\E F^2}.
\end{equation}
\Needspace{11\baselineskip}
\item Suppose that $\varepsilon\in\{\pm1\}$ and a nonempty subset $T\subseteq\F_p$
satisfy $\varepsilon\xi(t)>0$ for every $t\in T$. Set
\[
 U_1=\sum_{t\in T}\varepsilon\xi(t),\qquad
 U_2=\sum_{t\in T}\xi(t)^2,\qquad
 \gamma=\frac{U_2}{U_1}.
\]
Then there is a probability measure $\lambda$ on $\F_p$ such that
\[
 Z(y)=\sum_{s\in \F_p}\lambda(s)(1+\varepsilon\chi_p(y-s))\ge\frac12\qquad(y\in\F_p)
\]
and
\begin{equation}\label{eq:local-entropy-principle}
 \E_\nu\log Z\gg
 \min\left\{\gamma,\,U_2,\,
              \gamma\sqrt{\frac{U_2}{W}}\right\}.
\end{equation}
\end{enumerate}
\end{lemma}

\begin{proof}
By the character-correlation identity~\eqref{eq:char-correlation}
and $\sum_{y\in\F_p}\nu(y)=1$,
\[
 W=\sum_{s\in\F_p}\xi(s)^2
 =(p-1)\sum_{y\in\F_p}\nu(y)^2
   -\sum_{y\ne z}\nu(y)\nu(z)=p\sum_{y\in\F_p}\nu(y)^2-1
  =\E(p\nu-1)^2.
\]
This proves identity~\eqref{eq:correlation-moment}.
Since $\E_\nu F-\E F=\E((p\nu-1)F)$,
inequality~\eqref{eq:change-measure} follows by Cauchy--Schwarz.

For part~(ii), for each $s,y\in \F_p$, define
\[
 \mu(s)=\frac{\varepsilon\xi(s)\mathbf1_T(s)}{U_1},
 \qquad
 g(y)=\sum_{s\in \F_p}\mu(s)\varepsilon\chi_p(y-s).
\]
It follows that $\mu$ is a probability measure and thus $|g|\le1$. By the
definition of $\xi$,
\begin{equation}\label{eq:g}
\E_\nu g=\sum_{s\in\F_p}\mu(s)\varepsilon\xi(s)
 =\frac1{U_1}\sum_{s\in T}\xi(s)^2=\gamma.
\end{equation}
Applying the character-correlation identity~\eqref{eq:char-correlation} once more, we obtain
\begin{equation}\label{eq:g^2}
 \E g^2
 =\sum_{s\in \F_p}\mu(s)^2-\frac1p
 =\frac{U_2}{U_1^2}-\frac1p
 <\frac{U_2}{U_1^2}=\frac{\gamma^2}{U_2}.
\end{equation}

Applying part~(i) to $F=g^2$ and using $|g|\le1$, 
we obtain from estimate~\eqref{eq:g^2} that
\begin{equation}\label{eq:nug^2}
 \E_\nu g^2
 \le \E g^2+\sqrt{W\,\E g^4} \le \E g^2+\sqrt{W\,\E g^2}
 \le \frac{\gamma^2}{U_2}+\gamma\sqrt{\frac W{U_2}}
 =:M_0.
\end{equation}
For $0<t\le1/4$, take $\lambda=(1-t)u_p+t\mu$, where $u_p$ denotes the uniform measure on $\F_p$. Since
$\sum_s\chi_p(y-s)=0$ for each $y\in\F_p$, we have
$Z(y)=1+tg(y)\ge3/4$. As
$\log(1+x)\ge x-x^2$ for $|x|\le1/2$, we obtain
\[
 \E_\nu\log Z\ge t\E_\nu g-t^2\E_\nu g^2
 \ge t\gamma-t^2M_0
\]
from estimates~\eqref{eq:g} and~\eqref{eq:nug^2}. Choosing $t=\min\{1/4,\gamma/(2M_0)\}$, we have
$tM_0\le\gamma/2$, and hence
\[
 \E_\nu\log Z
 \ge \frac{t\gamma}{2}
 =\min\left\{\frac{\gamma}{8},\frac{\gamma^2}{4M_0}\right\}
 \geq \frac{1}{8}\min\left\{\gamma,U_2,
                 \gamma\sqrt{\frac{U_2}{W}}\right\},
\]
where the last inequality follows from the expression for $M_0$
in estimate~\eqref{eq:nug^2} and the elementary bound
$1/(a+b)\ge\frac12\min\{1/a,1/b\}$ for $a,b>0$.
\end{proof}

\begin{proposition}\label{prop:local}
Let $p$ be an odd prime and let $\nu$ be a probability measure on
$\F_p$. With $\xi$ and $W$ as in
Lemma~\ref{lem:positive-correlation}, suppose $W>0$.
Then there are a sign $\varepsilon\in\{\pm1\}$ and a
probability measure $\lambda$ on $\F_p$ such that
\[
 Z(y)=\sum_{s\in \F_p}\lambda(s)(1+\varepsilon\chi_p(y-s))\ge\frac12
 \qquad(y\in\F_p)
\]
and
\begin{equation}\label{eq:local-entropy}
 \E_\nu\log Z\gg\min\left\{W,\sqrt{\frac Wp}\right\}.
\end{equation}
\end{proposition}

\begin{proof}
Choose $\varepsilon\in\{\pm1\}$ so that the set
$T=\{s:\varepsilon\xi(s)>0\}$ accounts for at least half of
$\sum_s\xi(s)^2$. Then $U_2\ge W/2$, while Cauchy--Schwarz gives
$U_1^2\le pU_2$, and hence $\gamma\ge\sqrt{U_2/p}\ge\sqrt{W/(2p)}$.
It also follows that $\gamma\sqrt{U_2/W}\ge\frac12\sqrt{W/p}$.
The conclusion now follows from inequality~\eqref{eq:local-entropy-principle}.
\end{proof}

\subsection{From local gains to quadratic intersections}
We now combine the local logarithmic gains and pass from congruences modulo a product of primes to an integral quadratic. 
\begin{proposition}\label{prop:global-amplification}
Let \(T\) be a nonempty subset of $[N]$.  Let \(\mathcal P\) be a finite set
of primes \(p\equiv3\pmod4\), set
\[
        \Delta=\prod_{p\in\mathcal P}p,
\]
and assume that \(\Delta\le \lfloor\sqrt N\rfloor/2\). For each
\(p\in\mathcal P\), let \(\nu_p\) be the residue distribution of \(T\) modulo
\(p\).  Suppose that there is a common sign \(\varepsilon\in\{+1,-1\}\) and,
for each \(p\in\mathcal P\), a probability measure \(\lambda_p\) on \(\F_p\)
such that
\[
        Z_p(y):=\sum_{s\in\F_p}\lambda_p(s)(1+\varepsilon\chi_p(y-s))
\]
satisfies \(Z_p(y)\ge1/2\) for every \(y\in\F_p\), and
\[
        \sum_{y\in\F_p}\nu_p(y)\log Z_p(y)\ge \gamma_p .
\]
Then there is an integral quadratic \(q(x)=m+\varepsilon x^2\) such that
\[
        |T\cap q(\Z)|
        \gg
        \frac{|T|}{\sqrt N}
        \exp\left(\sum_{p\in\mathcal P}\gamma_p\right).
\]
\end{proposition}

\begin{proof}
Define a probability measure $w$ on \(\Z/\Delta\Z\) by assigning
\[
        w(r)=\frac{\#\{a\in T:a\equiv r\pmod\Delta\}}{|T|}
\]
for each $r\in \Z/\Delta \Z$. By the Chinese remainder theorem, we identify \(\Z/\Delta\Z\) with $\prod_{p\in \mathcal{P}} \F_p$. Observe that for each $p\in \mathcal{P}$, the \(p\)-marginal of $w$ is \(\nu_p\). Choose $r\in \Z/\Delta \Z$ according to $w$. For each $p\in \mathcal{P}$, choose shifts \(s_p\bmod p\) independently according to \(\lambda_p\).  For a
shift vector \(\mathbf s=(s_p)_{p\in\mathcal P}\), set
\[
        K_{\varepsilon,\mathbf s}(r)
        =\prod_{p\in\mathcal P}(1+\varepsilon\chi_p(r_p-s_p)),
\]
where \(r_p\) denotes the image of \(r\) modulo \(p\).  

Observe that we have
\[
\E_{r} \E_{\mathbf s}K_{\varepsilon,\mathbf s}(r)
        =\E_{r}\prod_{p\in\mathcal P}Z_p(r_p)
        =\E_{r} \left(\exp\left(\sum_{p\in\mathcal P} \log Z_p(r_p)\right)\right).
\]
Since the exponential function is convex, Jensen's inequality gives
\[
        \E_{r} \E_{\mathbf s}K_{\varepsilon,\mathbf s}(r)
        \geq \exp\left(\E_{r} \left(\sum_{p\in\mathcal P} \log Z_p(r_p)\right)\right)=\exp\left(
        \sum_{p\in\mathcal P}\sum_{y\in\F_p}\nu_p(y)\log Z_p(y)
        \right)\ge
        \exp\left(\sum_{p\in\mathcal P}\gamma_p\right).
\]
Thus some deterministic shift vector \(\mathbf s=(s_p)_{p\in\mathcal P}\) satisfies 
\begin{equation}\label{eq:lbexp}
\E_{r} K_{\varepsilon,\mathbf s}(r) \geq G:=\exp\left(\sum_{p\in\mathcal P}\gamma_p\right).
\end{equation}

Choose an integer $S \in [0,\Delta)$ so that \(S\equiv s_p\pmod p\) for all \(p\in\mathcal P\).  Since \(\chi_p(-1)=-1\) whenever $p\in \mathcal{P}$, it follows from the Chinese Remainder Theorem that
\[
        K_{\varepsilon,\mathbf s}(r)
        =
        \#\{x\bmod\Delta:r\equiv S+\varepsilon x^2\pmod\Delta\}
\]
for each $r\in \Z/\Delta\Z$. Therefore, inequality~\eqref{eq:lbexp} implies that
\[
        \sum_{a\in T}
        \#\{x\bmod\Delta:a\equiv S+\varepsilon x^2\pmod\Delta\}
        \ge |T|G .
\]

It remains to lift this congruence count.  Since \(\Delta\le \lfloor\sqrt N\rfloor/2\), every residue class modulo
\(\Delta\) contains at least \(\lfloor\sqrt N\rfloor/(2\Delta)\) integers $x$ with \(1\le x\le \lfloor\sqrt N\rfloor\).  Hence
there are at least \(\lfloor\sqrt N\rfloor |T|G/(2\Delta)\) pairs \((a,x)\) with \(a\in T\), \(1\le x\le \lfloor\sqrt N\rfloor\), and $
        a\equiv S+\varepsilon x^2\pmod\Delta.$ Each such pair satisfies \(a=S+\varepsilon x^2+\tau\Delta\) for some integer
\(\tau\).  Since \(a\in[N]\), \(x^2\le N\), and \(0\le S<\Delta\), there are
only \(O(N/\Delta)\) possible values of \(\tau\).  Therefore, for some integer \(\tau\),
\[
        \#\{1\le x\le \lfloor\sqrt N\rfloor:S+\tau\Delta+\varepsilon x^2\in T\}
        \gg
        \frac{|T|\lfloor\sqrt N\rfloor}{N}G
        \gg
        \frac{|T|}{\sqrt N}G.
\]
The map \(x\mapsto S+\tau\Delta+\varepsilon x^2\) is injective on
\(1\le x\le \lfloor\sqrt N\rfloor\), for both signs \(\varepsilon\).  This proves the proposition.
\end{proof}

\begin{remark}\label{rem:relative-entropy}
Let $\nu_p$ be a probability measure on $\F_p$, and let
$Z_p$ be a weight as in Proposition~\ref{prop:local}.
Since $\E Z_p=1$, the measure $Z_pu_p$ is a probability
measure on $\F_p$, where $u_p$ denotes the uniform measure on $\F_p$.
Writing $D$ for relative entropy (also called
Kullback--Leibler divergence), we have
\[
 \E_{\nu_p}\log Z_p
 =D(\nu_p\|u_p)-D(\nu_p\|Z_pu_p).
\]
Thus a positive logarithmic gain corresponds to a
decrease in relative entropy: $Z_pu_p$ is closer to
$\nu_p$ than the uniform measure is, in the sense of
relative entropy. This gives a relative-entropy
interpretation of the logarithmic averaging in
Proposition~\ref{prop:global-amplification}.
\end{remark}

We now prove Theorem~\ref{thm:main}.

\begin{proof}[Proof of Theorem~\ref{thm:main}]
Let
\[
 \mathcal P_0=\left\{p:\frac18\log N<p\le\frac14\log N,\quad
 p\equiv3\pmod4\right\}.
\]
For each $p\in\mathcal P_0$, Cauchy--Schwarz, the hypothesis,
and identity~\eqref{eq:correlation-moment} give
\[
 W_p=p\sum_{y\in\F_p}\nu_{A,p}(y)^2-1
 \ge\frac{p}{|A_p|}-1
 \ge\frac{p}{p/2+C}-1
 \ge\frac12
\]
for sufficiently large $N$ in terms of $C$.

By Proposition~\ref{prop:local}, for each $p\in\mathcal P_0$
there are a sign $\varepsilon_p\in\{\pm1\}$ and a probability
measure $\lambda_p$ on $\F_p$ such that
\[
 Z_p(y)=\sum_{s\in\F_p}
 \lambda_p(s)(1+\varepsilon_p\chi_p(y-s))
 \ge\frac12
 \qquad(y\in\F_p),
\]
and
\[
 \E_{\nu_{A,p}}\log Z_p\gg p^{-1/2}.
\]

By pigeonholing the two possible signs, there are a sign
$\varepsilon\in\{\pm1\}$ and a subset
$\mathcal P\subseteq\mathcal P_0$ such that
$\varepsilon_p=\varepsilon$ for every $p\in\mathcal P$, and
\[
 \sum_{p\in\mathcal P}\frac1{\sqrt p}
 \ge\frac12\sum_{p\in\mathcal P_0}\frac1{\sqrt p}.
\]
By the prime number theorem in the progression $3\pmod4$,
\[
 \sum_{p\in\mathcal P_0}\frac1{\sqrt p}
 \asymp\frac{\sqrt{\log N}}{\log\log N},
 \qquad
 \prod_{p\in\mathcal P_0}p
 =N^{1/16+o(1)}
 \le\frac{\lfloor\sqrt N\rfloor}{2}.
\]
Since $\mathcal P\subseteq\mathcal P_0$, the same product bound holds for $\mathcal P$.

Proposition~\ref{prop:global-amplification}, applied to
$T=A$ and the prime set $\mathcal P$, now gives an integral
quadratic $q(x)=m+\varepsilon x^2$ such that
\[
 |A\cap q(\Z)|
 \gg\frac{|A|}{\sqrt N}
 \exp\left(c\sum_{p\in\mathcal P}p^{-1/2}\right)
 \gg\frac{|A|}{\sqrt N}
 \exp\left(c\frac{\sqrt{\log N}}{\log\log N}\right).
\]
Since $|A|\ge\alpha\sqrt N$, the stated conclusion follows
after decreasing $c$.
\end{proof}

We also record a consequence under a weaker residue restriction.

\begin{corollary}\label{cor:square-root-deficit}
There are absolute constants $0<\eta<1/8$ and $c>0$ such that,
for each $b>0$, there is $N_0=N_0(b)$ with the following property.
For $N\ge N_0$, set
\[
 \mathcal Q_N=\left\{p:\frac\eta2\log N<p\le\eta\log N,\quad
 p\equiv3\pmod4\right\}.
\]
If $A\subseteq[N]$ is nonempty and
\[
 |A_p|\le p-b\sqrt p\qquad(p\in\mathcal Q_N),
\]
then there is an integral quadratic $q(x)=m\pm x^2$ such that
\[
 |A\cap q(\Z)|\gg\frac{|A|}{\sqrt N}
 \exp\left(c\sqrt b\,\frac{(\log N)^{1/4}}{\log\log N}\right).
\]
The implied constant is absolute.
\end{corollary}

\begin{proof}
By identity~\eqref{eq:correlation-moment},
for each $p\in\mathcal Q_N$ we have
\[
 W_p\ge\frac{p}{|A_p|}-1\ge bp^{-1/2}.
\]
Thus Proposition~\ref{prop:local} gives a logarithmic gain
$\gg\sqrt b\,p^{-3/4}$ for sufficiently large $N$
depending on $b$. The rest follows as in the proof of
Theorem~\ref{thm:main}, since
\[
 \sum_{p\in\mathcal Q_N}p^{-3/4}
 \asymp\frac{(\log N)^{1/4}}{\log\log N},
\]
and our choice of $\eta$ ensures that the product of
these primes is at most $\lfloor\sqrt N\rfloor/2$.
\end{proof}

\section{Truncated divisor-sum amplification}\label{sec:product}
The proof in Section~\ref{sec:local} combines local information
using the Chinese remainder theorem, but requires the product
of the selected primes to be at most $\sqrt N$. This prevents us from using all the primes in the larger ranges considered in Theorem~\ref{thm:largerrange}. To overcome this difficulty, we use the square of a truncated divisor sum, allowing us to combine information from more primes while controlling averages over integer intervals.

\subsection{A local dichotomy}
Let $\E$ denote uniform averaging over $\F_p$.

We establish a local dichotomy that will allow us to combine two different arguments. The parameter $V$ controls the size of the logarithmic gain in the first alternative and the moment bounds in the second alternative. Its value will be chosen in Section~\ref{sec:oneset}.

\begin{proposition}\label{prop:local-dichotomy}
There are absolute constants $c>0$ and $C_0\ge1$ with the following
property. Let $p$ be an odd prime, $1/p\le\delta<1$, and let $\nu$ be a
probability measure on $\F_p$ supported on at most $(1-\delta)p$
points. Suppose that
\[
 \max\{2,\delta^{-1/2}\}\le V\le\sqrt p.
\]
Then at least one of the following alternatives holds.
\begin{enumerate}[(i)]
\item There are a sign $\varepsilon\in\{\pm1\}$ and a probability
measure $\lambda$ on $\F_p$ such that
\[
 Z(y)=\sum_{s\in\F_p}\lambda(s)
       (1+\varepsilon\chi_p(y-s))\ge\frac12
 \qquad(y\in\F_p),
\]
and
\begin{equation*}
 \E_\nu\log Z\ge c\frac{\sqrt\delta}{V}.
\end{equation*}
\item There is a real function $\psi$ on $\F_p$ satisfying
$\E\psi=0$, $\E\psi^2\le1$, and $\|\psi\|_\infty\le\sqrt p$.
For $j=1,2$, define
\[
 \Psi_j(a)=\E_s\psi(s)^j\chi_p(a-s).
\]
Then
\begin{equation}\label{eq:local-dichotomy-bounds}
 \E_\nu \Psi_1\ge c\sqrt{\frac\delta p},\qquad
 \|\Psi_1\|_\infty\le C_0\frac V{\sqrt p},\qquad
 \E_\nu\left(\Psi_1^2+\frac{V^2}{p}\Psi_2^2\right)
 \le C_0\frac{\sqrt\delta}{V}.
\end{equation}
\end{enumerate}
\end{proposition}

\begin{proof}
Let $\Gamma=\supp\nu$ and $\delta_0=1-|\Gamma|/p\ge\delta$, and let $\xi,W$ be as in
Lemma~\ref{lem:positive-correlation}:
\[
 \xi(s)=\sum_{y\in \F_p}\nu(y)\chi_p(y-s),\qquad
 W=\sum_{s\in \F_p}\xi(s)^2.
\]
Identity~\eqref{eq:correlation-moment} and Cauchy--Schwarz give
\begin{equation}\label{eq:Wdelta}
 W=p\sum_{y\in\F_p}\nu(y)^2-1\ge \frac{p}{|\Gamma|}-1\ge\frac{\delta_0}{1-\delta_0}\ge\delta_0\ge\delta.
\end{equation}

If $W>\delta p/V^2$, then Proposition~\ref{prop:local} gives the required $Z$ with
\[
 \E_\nu\log Z\gg\min\left\{W,\sqrt{\frac Wp}\right\}
 \gg\min \left\{W,\frac{\sqrt\delta}{V}\right\}\geq \frac{\sqrt\delta}{V}
\]
so that (i) holds, where in the last step we used $W\geq \delta$ in inequality~\eqref{eq:Wdelta} and the assumption $V\ge\delta^{-1/2}$. 

Next, suppose that
\[
 \sum_{|\xi(s)|>\sqrt W/V}\xi(s)^2>\frac{\delta_0}{256}.
\]
Choose a sign $\varepsilon$ so that at least half this sum comes
from the set $T=\{s:\varepsilon\xi(s)>\sqrt W/V\}$. As in the notation of Lemma~\ref{lem:positive-correlation}(ii), set
\[
 U_1=\sum_{t\in T}\varepsilon\xi(t),\qquad
 U_2=\sum_{t\in T}\xi(t)^2,\qquad
 \gamma=\frac{U_2}{U_1}.
\]
Then we have
\begin{equation}\label{eq:U2W}
 U_2\ge\frac{\delta_0}{512},\qquad
 \gamma=\frac{U_2}{U_1}\ge\frac{\sqrt W}{V}.
\end{equation}
Since $W\ge\delta_0\ge\delta$ in inequality~\eqref{eq:Wdelta} and $V\ge\delta^{-1/2}$ by assumption,
we have $U_2\ge\sqrt\delta/(512V)$. The estimates in~\eqref{eq:U2W} imply that
\[
 \gamma\sqrt{\frac{U_2}{W}}=\frac{\gamma}{\sqrt{W}} \sqrt{U_2}
 \ge\frac{\sqrt{\delta_0}}{\sqrt{512}\,V}
 \ge\frac{\sqrt\delta}{\sqrt{512}\,V}.
\]
Thus, Lemma~\ref{lem:positive-correlation}(ii) gives the required $Z$ with
\[
 \E_\nu\log Z\gg
 \min\left\{\gamma,\,U_2,\,
              \gamma\sqrt{\frac{U_2}{W}}\right\}\gg \frac{\sqrt{\delta}}{V},
\]
so that (i) holds.

For the rest of the proof, we assume
\begin{equation}\label{eq:small-energy-tail}
 W\le\frac{\delta p}{V^2},\qquad
 \sum_{|\xi(s)|>\sqrt W/V}\xi(s)^2\le\frac{\delta_0}{256}.
\end{equation}

Our goal is to construct a function $\psi$ such that  estimates~\eqref{eq:local-dichotomy-bounds} hold. We first seek a function $f$ such that $\E_s f(s)\chi_p(a-s)>0$ for all $a\in \Gamma$. Motivated by the character-correlation
identity~\eqref{eq:char-correlation}, define \[
 f(s)=\frac1{\sqrt{\delta_0p}}
      \sum_{a\in\Gamma}\chi_p(a-s).
\]
Indeed, we have
\begin{equation}\label{eq:ef}
 \E_s f(s)\chi_p(a-s)
 =\frac{\mathbf1_\Gamma(a)-(1-\delta_0)}{\sqrt{\delta_0 p}}.
\end{equation}
Since $\nu$ is supported on $\Gamma$,
\begin{equation}\label{eq:Efxi}
 \E_s f(s)\xi(s)=\sqrt{\frac{\delta_0}{p}}.    
\end{equation}
The same identity~\eqref{eq:char-correlation} also gives
\begin{equation}\label{eq:ff}   
 \E f=0,\qquad \E f^2=\frac{1}{\delta_0 p^2} (p|\Gamma|-|\Gamma|^2)=\frac{1}{\delta_0}\big((1-\delta_0)-(1-\delta_0)^2\big)=1-\delta_0.
\end{equation}

Since $\sum_{a\in\F_p}\chi_p(a-s)=0$, we have
\[
 f(s)=-\frac1{\sqrt{\delta_0 p}}
       \sum_{a\notin \Gamma}\chi_p(a-s),
 \qquad \|f\|_\infty\le\sqrt{\delta_0 p}.
\]

Next we truncate $f$ to control the moment estimates in~\eqref{eq:local-dichotomy-bounds}.  Set
\[
 H=\frac{8\sqrt{pW}}{V\sqrt{\delta_0}},\qquad
 \widetilde f=\max\{-H,\min\{f,H\}\}.
\]
The height $H$ is chosen so that
$(\sqrt W/V)/H=\frac18\sqrt{\delta_0/p}$.
Since $|f-\widetilde f|\le f^2/H$ and
$|f-\widetilde f|\le|f|$, we have
\[
 \E|f-\widetilde f|\le\frac1H,\qquad
 \E(f-\widetilde f)^2\le1.
\]

We split the average according to whether
$|\xi(s)|\le\sqrt W/V$ or $|\xi(s)|>\sqrt W/V$.
Using Cauchy--Schwarz and the second inequality
in~\eqref{eq:small-energy-tail}, we obtain
\begin{align}
 \bigl|\E_s(f(s)-\widetilde f(s))\xi(s)\bigr|
 &\le\frac{\sqrt W}{V}\E|f-\widetilde f|
   +\left(\frac1p\sum_{|\xi(s)|>\sqrt W/V}\xi(s)^2\right)^{1/2} \notag\\
 &\le\left(\frac18+\frac1{16}\right)\sqrt{\frac{\delta_0}{p}}
 \le\frac14\sqrt{\frac{\delta_0}{p}} \label{eq:3/4}.
\end{align}
Define
\[
 \psi=\frac{\widetilde f-\E\widetilde f}{4}.
\]
Then $\E\psi=0$. Also, observe that $|\widetilde f|\le|f|\le\sqrt{\delta_0 p}$. Thus $\|\psi\|_\infty\le\sqrt p$, and inequality~\eqref{eq:ff} implies that
\[
\E\psi^2\leq \frac{\E \widetilde f^2}{16} \leq \frac{\E f^2}{16}\le \frac{1}{16}.
\]

Now
\[
\Psi_1(a)=\E_s\psi(s)\chi_p(a-s)=\frac{1}{4}\E_s \widetilde f(s)\chi_p(a-s).
\]
Therefore, by estimates~\eqref{eq:Efxi} and~\eqref{eq:3/4},
\[
 \E_\nu \Psi_1=\frac14\E_s\widetilde f(s)\xi(s)
 \ge\frac3{16}\sqrt{\frac{\delta_0}{p}}
 \ge\frac3{16}\sqrt{\frac\delta p}.
\]
By equation~\eqref{eq:ef} and $\E|f-\widetilde f|\le1/H$, we have
\[
 \|\Psi_1\|_\infty
 \le\frac14\left(\frac1{\sqrt{\delta_0 p}}+\frac1H\right)
 \ll\frac{1}{\sqrt{\delta_0 p}}+\frac{V\sqrt{\delta_0}}{\sqrt{pW}}\ll \frac V{\sqrt p},
\]
where in the last step we used $W\ge\delta_0\ge\delta$ in inequality~\eqref{eq:Wdelta} and $V\ge\delta^{-1/2}$ by assumption. 
We will also use
\[
 \|\psi\|_\infty^2\ll H^2
 \ll\frac{pW}{\delta_0 V^2}.
\]

The character-correlation
identity~\eqref{eq:char-correlation} gives
\[
 \E \Psi_1^2\le \frac{\E\psi^2}{p}\le\frac1p,\qquad
 \E \Psi_2^2\le\frac{\E\psi^4}{p}\le \frac{\|\psi\|_\infty^2 \E \psi^2}{p}  \ll\frac W{\delta_0 V^2}.
\]
Moreover, $\|\Psi_2\|_\infty\le\E\psi^2\le1$.
Applying Lemma~\ref{lem:positive-correlation}(i) to $\Psi_1^2$
and $\Psi_2^2$, respectively, yields
\[
\begin{aligned}
 \E_\nu \Psi_1^2
 &\le\E \Psi_1^2+\|\Psi_1\|_\infty\sqrt{W\E \Psi_1^2}
 \ll\frac1p+\frac{\sqrt W\,V}{p}
 \ll\sqrt{\frac\delta p},\\
 \E_\nu \Psi_2^2
 &\le\E \Psi_2^2+\|\Psi_2\|_\infty\sqrt{W\E \Psi_2^2}
 \ll\frac W{\delta_0 V^2}+\frac W{\sqrt{\delta_0} V}
 \ll\frac{\sqrt\delta\,p}{V^3}.
\end{aligned}
\]
In the last steps we used the bounds in~\eqref{eq:small-energy-tail},
$V\sqrt{\delta_0}\ge1$, and $\delta p\ge1$.
Finally, $V\le\sqrt p$ implies
\[
 \E_\nu\left(\Psi_1^2+\frac{V^2}{p}\Psi_2^2\right)
 \ll\sqrt{\frac\delta p}+\frac{\sqrt\delta}{V}
 \ll\frac{\sqrt\delta}{V}.
\]
This proves the bounds in~\eqref{eq:local-dichotomy-bounds}.
\end{proof}

\subsection{Amplification with a truncated divisor sum}\label{sec4.2}
We begin with an elementary weighted-average argument. Let
$T\subseteq[N]$ be nonempty and fix a sign $\varepsilon\in\{\pm1\}$.
For any nonnegative kernel $K:\Z\to\mathbb R_{\ge0}$, we have
\begin{equation}\label{eq:weighted-pigeonhole}
 \sum_{a\in T}\sum_{1\le x\le\sqrt N}K(a-\varepsilon x^2)
 \le
 \max_{m\in\Z}\left|T\cap\{m+\varepsilon x^2:x\in\Z\}\right| \cdot
 \sum_{m=-N}^{2N}K(m).
\end{equation}
Indeed, for each $m$, the number of pairs $(a,x)$ with
$a=m+\varepsilon x^2$ and $1\le x\le\sqrt N$ is at most
$|T\cap\{m+\varepsilon x^2:x\in\Z\}|$.
Taking $K\equiv1$ gives the elementary bound
\[
 \max_{m\in\Z}\left|T\cap\{m+\varepsilon x^2:x\in\Z\}\right|
 \gg\frac{|T|}{\sqrt N}.
\]

To improve this bound, we seek a kernel for which the left-hand side
of inequality~\eqref{eq:weighted-pigeonhole} is large relative to
$\sum_{m=-N}^{2N}K(m)$. Following the idea of Selberg's sieve, we
use the square of a truncated divisor sum built from the functions
in Proposition~\ref{prop:local-dichotomy}(ii). Squaring makes the
kernel nonnegative, while truncation allows us to estimate its
averages over integer intervals.

\begin{proposition}\label{prop:truncated-weight}
There are absolute constants $c>0$ and $N_0$ with the following
property. Let $N\ge N_0$, $0<\delta<1$, and let $T\subseteq[N]$ be
nonempty. Let $\mathcal P$ be a nonempty set of primes
$p\equiv3\pmod4$ with $p\le X\le(\log N)^2$, and set
\[
 S=\sum_{p\in\mathcal P}\frac1p.
\]
Suppose that $V$ satisfies
\begin{equation}\label{eq:truncation-conditions}
 \max\{2,\delta^{-1/2}\}\le V,\qquad
 V^2\le\min_{p\in\mathcal P}p,\qquad
 V^2S\le\frac{\log N}{\log X},
\end{equation}
and that the second alternative of
Proposition~\ref{prop:local-dichotomy} holds for the residue distribution
$\nu_{T,p}$ of $T$ at every $p\in\mathcal P$, with parameters $\delta$ and $V$.
Then there is an integral quadratic $q(x)=m\pm x^2$ such that
\begin{equation}\label{eq:truncated-weight-bound}
 |T\cap q(\Z)|\gg\frac{|T|}{\sqrt N}
 \exp(c\sqrt\delta\,VS).
\end{equation}
\end{proposition}

Fix the parameters in the proposition. 
For each $p\in\mathcal P$, choose $\psi_p$ from
Proposition~\ref{prop:local-dichotomy}(ii),
let $\Psi_{1,p},\Psi_{2,p}$ be the corresponding
functions defined there, and set
\[
\sigma_p^2=\E\psi_p^2, \qquad t_p=\frac{\eta V}{\sqrt p},
\]
where $\eta>0$ is a sufficiently small absolute constant. By
Proposition~\ref{prop:local-dichotomy},
$\E_{\nu_{T,p}}(t_p\Psi_{1,p})\gg\eta\sqrt\delta\,V/p$,
while $t_p\|\Psi_{1,p}\|_\infty\ll\eta V^2/p\le\eta$.
Define
\[
 Q=\prod_{p\in\mathcal P}p,\qquad
 R_\varepsilon(s)=\prod_{p\in\mathcal P}
             (1+\varepsilon t_p\psi_p(s)).
\]
We regard each $\psi_p$ as a function modulo $Q$ by the Chinese
remainder theorem. Set
\[
 \mathcal C=\E_{s\bmod Q}R_\varepsilon(s)^2,\qquad
 F_\varepsilon(a)=\E_{x\bmod Q}
             R_\varepsilon(a-\varepsilon x^2)^2.
\]

\begin{lemma}\label{lem:full-weight}
With the preceding notation, there is a sign
$\varepsilon\in\{\pm1\}$ such that
\begin{equation}\label{eq:full-weight-gain}
 \E_TF_\varepsilon
 \ge\mathcal C\exp(c\sqrt\delta\,VS)
\end{equation}
for an absolute constant $c>0$.
\end{lemma}

\begin{proof}
Since $\E\psi_p=0$, the Chinese remainder theorem gives
\[
 \mathcal C
 =\prod_{p\in\mathcal P}(1+t_p^2\sigma_p^2)\ge1.
\]
In particular, $\mathcal C$ is independent of $\varepsilon$.

For $p\equiv3\pmod4$ and any function $F$ on $\F_p$,
\begin{equation}\label{eq:quadratic-averaging}
 \E_{x\bmod p}F(a-\varepsilon x^2)
 =\E_{s\bmod p}F(s)
       \bigl(1+\varepsilon\chi_p(a-s)\bigr).
\end{equation}
Indeed, the number of solutions to
$a-\varepsilon x^2=s$ modulo $p$ is
$1+\varepsilon\chi_p(a-s)$, since $\chi_p(-1)=-1$.
Applying this identity to $F=\varepsilon\psi_p$
and $F=\psi_p^2$, and using $\E\psi_p=0$, gives
\[
 \E_{x\bmod p}
   \bigl(\varepsilon\psi_p(a-\varepsilon x^2)\bigr)
 =\Psi_{1,p}(a),\qquad
 \E_{x\bmod p}\psi_p(a-\varepsilon x^2)^2
 =\sigma_p^2+\varepsilon\Psi_{2,p}(a).
\]
Expanding $(1+\varepsilon t_p\psi_p)^2$ and
applying the Chinese remainder theorem, we obtain
\[
 F_\varepsilon(a)
 =\prod_{p\in\mathcal P}
 \bigl(1+2t_p\Psi_{1,p}(a)
       +t_p^2(\sigma_p^2+\varepsilon\Psi_{2,p}(a))
 \bigr).
\]

For each $p\in\mathcal P$, define functions
$\alpha_p,\beta_p$ on $\F_p$ by
\[
 \alpha_p(a)
 =\frac{2t_p\Psi_{1,p}(a)}
        {1+t_p^2\sigma_p^2},
 \qquad
 \beta_p(a)
 =\frac{t_p^2\Psi_{2,p}(a)}
        {1+t_p^2\sigma_p^2}.
\]
Then
\[
 \frac{F_\varepsilon(a)}{\mathcal C}
 =\prod_{p\in\mathcal P}
   (1+\alpha_p(a)+\varepsilon\beta_p(a)).
\]
We consider both signs $\varepsilon=\pm1$ so that
the terms linear in $\beta_p$ cancel when we
average the logarithms.

By Proposition~\ref{prop:local-dichotomy},
\begin{equation}\label{eq:12}
 \|\Psi_{1,p}\|_\infty\ll\frac{V}{\sqrt p},
 \qquad
 \|\Psi_{2,p}\|_\infty
 \le\E\psi_p^2=\sigma_p^2\le1.
\end{equation}
Since $t_p=\eta V/\sqrt p$ and $V^2\le p$,
it follows that 
\[
 \|\alpha_p\|_\infty
 \ll\frac{t_pV}{\sqrt p}\ll\eta,
 \qquad
 \|\beta_p\|_\infty\le t_p^2\le\eta^2.
\]
Thus, for sufficiently small $\eta$, we have $
 \|\alpha_p\|_\infty+\|\beta_p\|_\infty
 \le\frac12.$
In particular, all the factors
$1+\alpha_p\pm\beta_p$ are positive.
Applying $\log(1+u)\ge u-u^2$ for
$|u|\le1/2$ to both signs and averaging,
we obtain, for every $a\in\F_p$,
\[
\frac12\log\bigl(
  (1+\alpha_p(a)+\beta_p(a))
  (1+\alpha_p(a)-\beta_p(a))\bigr)\ge
  \alpha_p(a)-\alpha_p(a)^2-\beta_p(a)^2.
\]

We now average with respect to the residue
distribution $\nu_{T,p}$. Observe that $t_p^2\sigma_p^2\le t_p^2\leq \eta^2\leq 1$. By
Proposition~\ref{prop:local-dichotomy}, we have
\[
 \E_{\nu_{T,p}}\alpha_p
=\frac{2t_p}{1+t_p^2\sigma_p^2}
\E_{\nu_{T,p}}\Psi_{1,p}\gg t_p
\E_{\nu_{T,p}}\Psi_{1,p}\gg \eta\frac{\sqrt\delta\,V}{p},
\]
\[
 \E_{\nu_{T,p}}(\alpha_p^2+\beta_p^2)\ll t_p^2\E_{\nu_{T,p}} \Psi_{1,p}^2+t_p^4 \E_{\nu_{T,p}} \Psi_{2,p}^2
 \ll\frac{\eta^2 V^2}{p}
   \E_{\nu_{T,p}}\left(
      \Psi_{1,p}^2+\frac{V^2}{p}\Psi_{2,p}^2
   \right)
 \ll\eta^2\frac{\sqrt\delta\,V}{p}.
\]
Thus, choosing $\eta$ sufficiently small, we have
\[
\E_{\nu_{T,p}} \frac12\log\bigl(
  (1+\alpha_p+\beta_p)
  (1+\alpha_p-\beta_p)\bigr)\gg \eta\frac{\sqrt\delta\,V}{p}.
\]

Since averaging a function modulo $p$ over $T$ agrees with averaging it against $\nu_{T,p}$, summing over $p\in\mathcal P$ yields 
\[
 \E_T\log
 \frac{\sqrt{F_+F_-}}{\mathcal C}
 \ge c\sqrt\delta\,VS
\]
for an absolute constant $c>0$.
Finally, Jensen's inequality gives
\[
 \max_{\varepsilon=\pm1}\E_TF_\varepsilon
 \ge\frac{\E_TF_++\E_TF_-}{2}
 \ge\E_T\sqrt{F_+F_-}
 \ge \mathcal C\exp\left(
    \E_T\log\frac{\sqrt{F_+F_-}}{\mathcal C}
    \right)
 \ge\mathcal C\exp(c\sqrt\delta\,VS),
\]
as required.
\end{proof}

Fix the sign $\varepsilon$ supplied by Lemma~\ref{lem:full-weight}. One might try to use $K(s)=R_\varepsilon(s)^2$ directly in inequality~\eqref{eq:weighted-pigeonhole}. 
However, Lemma~\ref{lem:full-weight} gives a lower
bound for $\E_TF_\varepsilon$, where $F_\varepsilon(a)$
is defined by averaging over all $x\bmod Q$,
whereas the left-hand side of
inequality~\eqref{eq:weighted-pigeonhole}
involves only $1\le x\le\sqrt N$.
Since $R_\varepsilon$ is $Q$-periodic, comparing the finite sum with the complete average by periodicity introduces an error of size $O(Q\|R_\varepsilon\|_\infty^2)$, which need not be small since $Q$ may be much larger than $\sqrt N$.

To overcome this difficulty, we truncate the divisor
expansion of $R_\varepsilon$. Let
$D=\lfloor N^{1/100}\rfloor$ and define
\[
 R_D(s)=\sum_{\substack{d\mid Q\\d\le D}}
          \prod_{p\mid d}(\varepsilon t_p\psi_p(s)).
\]
Set
\[
 F_D(a)=\E_{x\bmod Q}R_D(a-\varepsilon x^2)^2,
 \qquad
 \mathcal C_D=\E_{s\bmod Q}R_D(s)^2.
\]
Each summand in $R_D$ is $d$-periodic for some
$d\le D$. Consequently, each term in the expansion
of $R_D^2$ has period dividing
$\operatorname{lcm}(d_1,d_2)\le D^2$.
This allows us to compare complete averages with
sums over integer intervals term by term, with
errors depending on $D$ rather than $Q$.
Moreover, Rankin's trick shows that
$F_D(a)$ approximates $F_\varepsilon(a)$ uniformly
in $a$, so the gain from Lemma~\ref{lem:full-weight}
is preserved. The next lemma establishes these
estimates.

\begin{lemma}\label{lem:truncation-comparison}
With the sign $\varepsilon$ fixed as above, there is
an absolute constant $c>0$ such that, for every integer $a$,
\begin{equation}\label{eq:truncation-error}
 |F_\varepsilon(a)-F_D(a)|
 \le \mathcal C\exp\left(-c\frac{\log N}{\log X}\right).
\end{equation}
Moreover,
\begin{align}
 \sum_{1\le x\le\sqrt N}R_D(a-\varepsilon x^2)^2
 &=\lfloor\sqrt N\rfloor F_D(a)+O(D^5),
 \label{eq:interval-numerator}\\
 \sum_{m=-N}^{2N}R_D(m)^2
 &=(3N+1)\mathcal C_D+O(D^5),
 \label{eq:interval-denominator}
\end{align}
and $1\le\mathcal C_D\le\mathcal C$.
All implied constants are absolute.
\end{lemma}

\begin{proof}
Expanding $R_\varepsilon^2$ and $R_D^2$ over pairs
of squarefree divisors of $Q$, and using
identity~\eqref{eq:quadratic-averaging}, we obtain
\[
F_\varepsilon(a)-F_D(a)
=\sum_{\substack{d_1,d_2\mid Q\\
                    \max\{d_1,d_2\}>D}}
\prod_{\substack{p\mid d_1d_2\\
                 p\nmid(d_1,d_2)}}
t_p\Psi_{1,p}(a) \cdot 
\prod_{p\mid(d_1,d_2)}
t_p^2\bigl(\sigma_p^2+
           \varepsilon\Psi_{2,p}(a)\bigr).
\]
Here $(d_1,d_2)$ denotes the greatest common divisor.
Since
$\|\Psi_{2,p}\|_\infty\le\sigma_p^2$ by inequality~\eqref{eq:12},
we have $|\sigma_p^2+\varepsilon\Psi_{2,p}(a)|\le 2\sigma_p^2.$
We now apply Rankin's trick to estimate the sum.
Set $\kappa=1/\log X$. Then
\[
 \mathbf1_{\max\{d_1,d_2\}>D}
 \le D^{-\kappa}(d_1^\kappa+d_2^\kappa).
\]
It follows that
\[
 |F_\varepsilon(a)-F_D(a)|
 \le 2D^{-\kappa}
 \prod_{p\in\mathcal P}
 \bigl(1+(1+p^\kappa)t_p|\Psi_{1,p}(a)|+2p^\kappa t_p^2\sigma_p^2\bigr).
\]
By the bounds in
\eqref{eq:12},
\[
 \sum_{p\in\mathcal P}
 t_p|\Psi_{1,p}(a)|
 \ll\eta V^2 \sum_{p\in \mathcal P}\frac{1}{p}=\eta V^2S,
 \qquad
 \sum_{p\in\mathcal P}t_p^2\sigma_p^2
 \ll\eta^2V^2\sum_{p\in \mathcal P}\frac{1}{p}=\eta^2V^2S.
\]
Since $p^\kappa\le e$ for each $p\in \mathcal P$ and $1+u\le e^u$
for $u\ge0$, it follows that
\[
 |F_\varepsilon(a)-F_D(a)|
 \le 2\exp\left(
   -\frac{\log D}{\log X}
   +O(\eta V^2S)\right).
\]
Recall that
\[
 D=\lfloor N^{1/100}\rfloor,
 \qquad
 V^2S\le\frac{\log N}{\log X}.
\]
Thus, by choosing $\eta$ sufficiently
small, we obtain
\[
 |F_\varepsilon(a)-F_D(a)|
 \le\exp\left(
   -c\frac{\log N}{\log X}\right)
\]
for some absolute constant $c>0$.
Since $\mathcal C\ge1$, this proves
estimate~\eqref{eq:truncation-error}.

We next prove estimates \eqref{eq:interval-numerator} and
\eqref{eq:interval-denominator}. Since $\|t_p\psi_p\|_\infty\le\eta V\le\sqrt p$,
each term in the expansion of $R_D^2$, indexed by
$d_1,d_2\le D$, has period dividing
$\operatorname{lcm}(d_1,d_2)\le D^2$ and
absolute value at most $\sqrt{d_1d_2}\le D$.
The same bounds hold after replacing $s$ by
$a-\varepsilon x^2$.
For a $q$-periodic function $h$ and an interval $I$
of consecutive integers, we have
\[
 \sum_{n\in I}h(n)
 =|I|\E_{r\bmod q}h(r)+O(q\|h\|_\infty).
\]
Thus each divisor pair contributes an error
$O(D^3)$, and there are at most $D^2$ pairs.
This proves \eqref{eq:interval-numerator} and
\eqref{eq:interval-denominator}.

Finally, since $\E\psi_p=0$, the terms indexed
by distinct divisors have mean zero. Hence
\[
 \mathcal C_D
 =\sum_{\substack{d\mid Q\\d\le D}}
   \prod_{p\mid d}t_p^2\sigma_p^2
\]
and
\[
 1\le\mathcal C_D
 \le\prod_{p\in\mathcal P}(1+t_p^2\sigma_p^2)
 =\mathcal C.\qedhere
\]
\end{proof}

\begin{proof}[Proof of Proposition~\ref{prop:truncated-weight}]
Fix the sign $\varepsilon$ supplied by
Lemma~\ref{lem:full-weight}.
By Lemmas~\ref{lem:full-weight}
and~\ref{lem:truncation-comparison},
\[
 \E_TF_D
 \ge \mathcal C\left(
   \exp(c\sqrt\delta\,VS)
   -\exp\left(-c'\frac{\log N}{\log X}\right)
   \right)\gg \mathcal C\exp(c\sqrt\delta\,VS),
\]
where $c,c'>0$ are absolute constants.
Here the second term is negligible since
$\log N/\log X\to\infty$.

Using estimate~\eqref{eq:interval-numerator}
and $D^5=o(\sqrt N)$, we obtain
\[
 \sum_{a\in T}\sum_{1\le x\le\sqrt N}
   R_D(a-\varepsilon x^2)^2
 =\lfloor\sqrt N\rfloor
          |T|\E_TF_D+O(|T|D^5)
 \gg |T|\sqrt N\,\mathcal C
          \exp(c\sqrt\delta\,VS).
\]
Similarly, estimate~\eqref{eq:interval-denominator}
and $1\leq \mathcal C_D\le\mathcal C$ give
\[
 \sum_{m=-N}^{2N}R_D(m)^2
 =(3N+1)\mathcal C_D+O(D^5)
 \ll N\mathcal C.
\]

Applying inequality~\eqref{eq:weighted-pigeonhole}
with $K=R_D^2$, we conclude that
\[
 \max_{m\in\Z}
 \bigl|T\cap\{m+\varepsilon x^2:x\in\Z\}\bigr|
 \gg \frac{|T|}{\sqrt N}
       \exp(c\sqrt\delta\,VS).\qedhere
\]
\end{proof}

\section{Combining the two alternatives}\label{sec:oneset}

We now combine the two amplification arguments.

\begin{proposition}\label{prop:combined-estimate}
Fix $\eta_0>0$. There are constants $c=c(\eta_0)>0$ and
$N_0=N_0(\eta_0)$ such that, for all $N\ge N_0$ and $0<\delta<1$,
the following holds. Let
\[
 \mathcal P_0\subseteq
 \{p:Y<p\le X,\ p\equiv3\pmod4\},\qquad
 \eta_0\log N\le Y<X\le(\log N)^2,
\]
and suppose that $T\subseteq[N]$ is nonempty with
$|T_p|\le(1-\delta)p$ for every $p\in\mathcal P_0$. Set
\[
 S=\sum_{p\in\mathcal P_0}\frac1p.
\]
Then there is an integral quadratic $q(x)=m\pm x^2$ such that
\begin{equation}\label{eq:combined-estimate}
 |T\cap q(\Z)|\gg_{\eta_0}\frac{|T|}{\sqrt N}
 \exp\left(c\sqrt\delta\,
 \min\left\{\sqrt{\frac{S\log N}{\log X}},\sqrt Y\,S\right\}\right).
\end{equation}
\end{proposition}

\begin{proof}
If $S=0$, the conclusion follows from the elementary bound
$|T\cap q(\Z)|\gg |T|/\sqrt N$ noted after the weighted averaging
inequality~\eqref{eq:weighted-pigeonhole}. Suppose $S>0$, and set
\[
 V=\min\left\{\sqrt{\frac{\log N}{S\log X}},\sqrt Y\right\}.
\]
Then
\[
 V^2S\le\frac{\log N}{\log X},\qquad
 V^2\le Y<p\quad(p\in\mathcal P_0).
\]
The prime number theorem gives $S\ll_{\eta_0}1$. Thus the
same elementary bound suffices if $\sqrt\delta\,VS$ is bounded
by a sufficiently large constant depending only on $\eta_0$.
We may therefore assume that $\sqrt\delta\,VS$ is large.
Since $S\ll_{\eta_0}1$, our assumption that $\sqrt\delta\,VS$ is sufficiently large gives
$V\ge\max\{2,\delta^{-1/2}\}$. Moreover, $V^2<p$ implies
$1/p<1/V^2\le\delta$, so Proposition~\ref{prop:local-dichotomy} applies with $\nu=\nu_{T,p}$ for every $p\in\mathcal P_0$.

Let $\mathcal L$ be the primes in $\mathcal P_0$ for which
part~(i) of Proposition~\ref{prop:local-dichotomy} holds.
For each $p\in\mathcal L$, choose a sign $\varepsilon_p$ and
weight $Z_p$ such that
\[
 \E_{\nu_{T,p}}\log Z_p\ge c\frac{\sqrt\delta}{V}.
\]
For every $p\in\mathcal P_0\setminus\mathcal L$,
part~(ii) supplies a function $\psi_p$ with the required bounds.
Set $k=\lfloor V^2S/8\rfloor$. Since $\sqrt\delta\,VS$ is
large, we may assume $V^2S\ge16$, so $k\ge V^2S/16$.

\emph{Case 1: $|\mathcal L|\ge2k$.}
At least $k$ primes in $\mathcal L$ have the same sign
$\varepsilon$. Select $k$ of them and denote this set by
$\mathcal P$. Their product satisfies
\[
 \prod_{p\in\mathcal P}p\le X^k
 \le\exp\left(\frac18V^2S\log X\right)
 \le N^{1/8}\le\frac{\lfloor\sqrt N\rfloor}{2}
\]
for sufficiently large $N$. The selected weights have the common
sign $\varepsilon$, and their residue distributions are
$\nu_{T,p}$. Proposition~\ref{prop:global-amplification}
therefore gives an integral quadratic $q(x)=a+\varepsilon x^2$
such that
\[
 |T\cap q(\Z)|\gg\frac{|T|}{\sqrt N}
 \exp\left(c\frac{\sqrt\delta\,k}{V}\right)
 \gg\frac{|T|}{\sqrt N}\exp(c'\sqrt\delta\,VS),
\]
where the last inequality uses $k\ge V^2S/16$.

\emph{Case 2: $|\mathcal L|<2k$.}
Since every prime in $\mathcal P_0$ exceeds $Y$,
\[
 \sum_{p\in\mathcal L}\frac1p
 \le\frac{2k}{Y}\le\frac{V^2S}{4Y}\le\frac S4.
\]
Thus $\mathcal P=\mathcal P_0\setminus\mathcal L$ has reciprocal
sum at least $3S/4$. The functions $\psi_p$ supplied by part~(ii),
together with the bounds on $V$ established above, allow us to apply
Proposition~\ref{prop:truncated-weight} to $\mathcal P$. We obtain
an integral quadratic $q(x)=a\pm x^2$ such that
\[
 |T\cap q(\Z)|\gg\frac{|T|}{\sqrt N}
 \exp\left(c\sqrt\delta\,V\sum_{p\in\mathcal P}\frac1p\right)
 \gg\frac{|T|}{\sqrt N}\exp(c'\sqrt\delta\,VS).
\]
Both cases prove estimate~\eqref{eq:combined-estimate}.
\end{proof}

\begin{proof}[Proof of Theorem~\ref{thm:largerrange}]
Let
\[
 \mathcal P_0=
 \left\{p:\frac{\theta\log N}{2}<p\le X,\ p\equiv3\pmod4\right\}.
\]
The prime number theorem in the progression $3\pmod4$ gives,
uniformly for $\theta\log N\le X\le(\log N)^2$,
\[
 S:=\sum_{p\in\mathcal P_0}\frac1p
 \asymp_\theta\frac{1+\log(X/(\theta\log N))}{\log\log N},
 \qquad \log X\asymp_\theta\log\log N.
\]
Since $Y=(\theta/2)\log N$ and $S\log X\gg_\theta1$,
we have
\[
 \min\left\{\sqrt{\frac{S\log N}{\log X}},
                  \sqrt Y\,S\right\}
 \asymp_\theta\sqrt{\frac{S\log N}{\log X}}
 \asymp_\theta\frac{\sqrt{\log N}}{\log\log N}
 \sqrt{1+\log\frac{X}{\theta\log N}}.
\]
The conclusion follows from
Proposition~\ref{prop:combined-estimate}. If $X=(\log N)^{1+\epsilon}$ with fixed $0<\epsilon\le1$,
then $1+\log(X/(\theta\log N))\asymp_{\theta,\epsilon}\log\log N$,
which proves the final assertion.
\end{proof}

\section{The two-set theorem and the inverse Goldbach application}\label{sec:twoset}

We begin with the following form of Gallagher's larger sieve \cite{Gallagher1971}.

\begin{lemma}\label{lem:larger-sieve}
Let $T\subseteq[N]$ be nonempty and let $2\le Q\le\sqrt N$. Then
\begin{equation}\label{eq:larger-sieve-moment}
 \sum_{p\le Q}\frac{\log p}{|T_p|}
 \le\log N+O\left(\frac Q{|T|}\right).
\end{equation}
\end{lemma}

\begin{proof}
For each prime $p$, let $m_p(r)$ count the elements of $T$
congruent to $r\pmod p$. By Cauchy--Schwarz, the number of
ordered pairs of distinct elements of $T$ congruent modulo $p$ is
\[
 \begin{aligned}
 \sum_{r\in\F_p}m_p(r)\bigl(m_p(r)-1\bigr)
 &=\sum_{r\in\F_p}m_p(r)^2-|T| \ge\frac{|T|^2}{|T_p|}-|T|.
 \end{aligned}
\]
On the other hand, for distinct $a,a'\in T$, we have
\[\sum_{p\mid a-a'}\log p\le\log|a-a'|\le\log N.\]
Summing over all such pairs gives
\[
 \begin{aligned}
 \sum_{p\le Q}\log p\left(\frac{|T|^2}{|T_p|}-|T|\right)
 &\le\sum_{\substack{a,a'\in T\\a\ne a'}}
       \sum_{\substack{p\le Q\\p\mid a-a'}}\log p \le |T|(|T|-1)\log N.
 \end{aligned}
\]
Dividing by $|T|^2$ and using $\sum_{p\le Q}\log p\ll Q$
proves the lemma.
\end{proof}

\begin{proof}[Proof of Theorem~\ref{thm:twoset-density}]
Set
\[
 Q=\min\{|A|,|B|,\sqrt N\}=\sqrt N\,e^{-t},
\]
where $t$ is as in the theorem. For $p\le Q$, define
\begin{equation}\label{eq:dp}
 d_p=\frac1{|A_p|}+\frac1{|B_p|}-\frac4{p+C}.
\end{equation}
Condition~\eqref{eq:pair} implies that $d_p\ge0$. Applying Lemma~\ref{lem:larger-sieve} to $A$ and $B$,
and using $Q\le\min(|A|,|B|)$, we have
\[
\sum_{p\leq Q} \frac{\log p}{|A_p|}\leq \log N+O(1), \quad  \sum_{p\leq Q} \frac{\log p}{|B_p|}\leq \log N+O(1).
\]
Since $\sum_{p\le Q}(\log p)/(p+C)=\log Q+O_C(1)$, we obtain from equation~\eqref{eq:dp} that
\begin{equation}\label{eq:larger-sieve-defect}
 \sum_{p\le Q}d_p\log p
 \le2\log N-4\log Q+O_C(1)=4t+O_C(1).
\end{equation}

We next find a set of primes on which both $A$
and $B$ occupy sufficiently many residue classes. Let
\[
 \mathcal P_0=\{p:\log N<p\le(\log N)^2,\ p\equiv3\pmod4\}.
\]
Since $t\le\sqrt{\log N}$, we have
$Q\ge\sqrt N e^{-\sqrt{\log N}}>(\log N)^2$ for sufficiently large $N$.
Thus $\mathcal P_0\subseteq\{p:p\le Q\}$, and the prime number theorem
in the progression $3\pmod4$ gives
\begin{equation}\label{eq:p0}
 \sum_{p\in\mathcal P_0}\frac1p=\frac{\log2}{2}+o(1).
\end{equation}
Choose a sufficiently large absolute constant $K_0$, and set
\[
 R=K_0\left(1+\frac{t}{\log\log N}\right),\qquad
 \mathcal P=\{p\in\mathcal P_0:|A_p|,|B_p|\ge p/R\}.
\]
For $p\in\mathcal P_0\setminus\mathcal P$, at least one of
$|A_p|,|B_p|$ is smaller than $p/R$, so $d_p\ge(R-4)/p$ by equation~\eqref{eq:dp}.
Hence inequality~\eqref{eq:larger-sieve-defect} implies
\begin{equation}\label{eq:p0p}
 \sum_{p\in\mathcal P_0\setminus\mathcal P}\frac1p
 \le\frac{4t+O_C(1)}{(R-4)\log\log N}.
\end{equation}
Taking $K_0$ sufficiently large and then $N$ sufficiently large
in terms of $C$, the right-hand side of inequality~\eqref{eq:p0p} is less than $\frac{\log2}{4}$.
Consequently, inequalities~\eqref{eq:p0} and~\eqref{eq:p0p} imply that
\[
 S:=\sum_{p\in\mathcal P}\frac1p\gg1.
\]
Moreover, for
$p\in\mathcal P$, condition~\eqref{eq:pair} yields
\[
 |A_p|,|B_p|\le p-\frac pR+C\le\left(1-\frac1{2R}\right)p,
\]
since $R\ll1+\sqrt{\log N}/\log\log N=o(\log N)$ and $p>\log N$.

Now apply Proposition~\ref{prop:combined-estimate} to $A$ and $B$
separately with
\[
 Y=\log N,\qquad X=(\log N)^2,\qquad
 \delta=\frac1{2R},\qquad \eta_0=1.
\]
Since
\[
 \min\left\{\sqrt{\frac{S\log N}{\log X}},\sqrt Y\,S\right\}
 \gg\sqrt{\frac{\log N}{\log\log N}},
\]
Proposition~\ref{prop:combined-estimate} gives, for each $T\in\{A,B\}$,
\[
 |T\cap q_T(\Z)|\gg_C\frac{|T|}{\sqrt N}
 \exp\left(c\sqrt{\frac{\log N}{R\log\log N}}\right).
\]
The result follows from $R\log\log N=K_0(\log\log N+t)$.
\end{proof}

\begin{proof}[Proof of Corollary~\ref{cor:goldbach-quadratic}]
Choose $T_0$ such that every element of
$(\mathcal A+\mathcal B)\cap[T_0,\infty)$ is prime. By
Elsholtz--Harper~\cite[Theorem~2.6]{ElsholtzHarper2015}, there is a constant
$c_*>0$ such that, for all sufficiently large $N$,
\[
        |\mathcal A\cap[N]|,\ |\mathcal B\cap[N]|
        \ge c_*\frac{\sqrt N}{\log N\log\log N}.
\]
Set
\[
        Y=\left\lfloor
        \frac{c_*}{4}\frac{\sqrt N}{\log N\log\log N}
        \right\rfloor.
\]
For sufficiently large $N$, each summand has at least $4Y$ elements
in $[N]$, at most $Y$ of which lie in $[Y]$. Thus we may choose
$A_N\subseteq\mathcal A\cap(Y,N]$ and $B_N\subseteq\mathcal B\cap(Y,N]$ with $|A_N|=|B_N|=Y.$ For every prime $p\le Y$, the residue sets $(A_N)_p$ and $-(B_N)_p$ are disjoint. Otherwise, $p$ would divide $a+b$ for some $a\in A_N$ and $b\in B_N$, although $a+b>2Y\ge T_0$ is a prime larger than $p$. Hence $|(A_N)_p|+|(B_N)_p|\le p$ for all primes $p\leq Y$. Since $|A_N|=|B_N|=Y$, we have
\[
 t=\log\frac{\sqrt N}{Y}=O(\log\log N)
 \le\sqrt{\log N}.
\]
Thus Theorem~\ref{thm:twoset-density} applies with $C=0$.

For each $T\in\{A_N,B_N\}$, we obtain
\[
        |T\cap q_T(\Z)|
        \gg
        \frac1{\log N\log\log N}
        \exp\left(c\sqrt{\frac{\log N}{\log\log N}}\right).
\]
Since $\log(\log N\log\log N)
=o(\sqrt{\log N/\log\log N})$, the asserted lower bound follows after decreasing $c$. The inclusions
$A_N\subseteq\mathcal A\cap[N]$ and $B_N\subseteq\mathcal B\cap[N]$
complete the proof.
\end{proof}

\section{A matching construction}\label{sec:upper}
We prove Theorem~\ref{thm:sharpness} by first constructing suitable
residue sets over each $\F_p$, then combining them by the Chinese
remainder theorem and randomly thinning the resulting set.
\subsection{Local residue sets} 
Set
\[
        g(X)=X(X-1)(X-2).
\] For each odd prime $p$, let $\chi_p$ be the quadratic character of
$\F_p$, extended by $\chi_p(0)=0$. For $p>3$, set
\[
        H_p^{(0)}=\{x\in\F_p:\chi_p(g(x))=1\}.
\]
The following lemma shows that these sets can be adjusted to have
cardinality exactly $(p-1)/2$ while retaining a uniform bound
on the number of solutions to $s+az^2\in H_p$.

\begin{lemma}\label{lem:unif_shift}
There is an absolute constant $C_0>0$ such that, for every odd prime
$p$, there is a set $H_p\subseteq\F_p$ with $|H_p|=(p-1)/2$ and
\[
        \#\{z\in\F_p:s+az^2\in H_p\}
        \le |H_p|(1+C_0p^{-1/2})
\]
for all $s\in\F_p$ and $a\in\F_p^\times$.
\end{lemma}

\begin{proof}
For $p=3$, take any singleton and enlarge $C_0$ if necessary.
Suppose $p>3$ and write $\chi=\chi_p$. Since
$g(X)=X(X-1)(X-2)$ is squarefree, the Weil bound
\cite[Theorem~5.41]{LN97} gives
\[
        |H_p^{(0)}|=\frac p2+O(\sqrt p).
\]

Fix $s\in\F_p$ and $a\in\F_p^\times$. We claim that
$g(s+aZ^2)$ is not a constant multiple of a square of a polynomial over $\F_p$; indeed, choose
$r\in\{0,1,2\}$ with $r\ne s$. Then $s+aZ^2-r$ has two distinct
simple roots in $\overline{\F_p}$, neither of which is a root of the
other two factors of $g(s+aZ^2)$. Thus $g(s+aZ^2)$ has a simple root
and hence is not a constant multiple of a square of a polynomial over $\F_p$. Applying the Weil bound again, we obtain
\[
        \sum_{z\in\F_p}\chi\big(g(s+az^2)\big)=O(\sqrt p)
\]
uniformly in $s$ and $a$. Since $g(s+aZ^2)$ has at most six zeros,
\[
        \#\{z\in\F_p:s+az^2\in H_p^{(0)}\}
        =\frac p2+O(\sqrt p)
\]
uniformly in $s$ and $a$. Adding or removing $O(\sqrt p)$ elements gives a set $H_p$
of size exactly $(p-1)/2$. Adding or removing one element changes
the number of solutions by at most two, so the bound
$p/2+O(\sqrt p)=|H_p|(1+O(p^{-1/2}))$ is preserved.
\end{proof}

\begin{remark}
Lemma~\ref{lem:unif_shift} also admits a conceptually different proof
using Spencer's classical discrepancy theorem~\cite[Theorem~1]{Spencer1985}. In its linear-form formulation, the theorem states that for any $n$ linear forms in $n$ variables, with coefficients of absolute value at most one, one can assign each variable a value in $\{-1,1\}$ so that every form has absolute value $O(\sqrt n)$.

Let $\chi$ be the quadratic character of $\F_p$. Apply Spencer's theorem to the $p+1$ forms
\[
        \sum_{x\in\F_p}v_x,\qquad
        \sum_{x\in\F_p}\chi(x-s)v_x\quad(s\in\F_p),
\]
after adjoining one dummy variable. We obtain signs
$\varepsilon_x\in\{-1,1\}$ such that
\[
        \left|\sum_x\varepsilon_x\right|\ll\sqrt p,\qquad
        \sup_{s\in\F_p}
        \left|\sum_x\varepsilon_x\chi(x-s)\right|\ll\sqrt p.
\]
Thus, for $H_0=\{x:\varepsilon_x=1\}$,
\[
        |H_0|=\frac p2+O(\sqrt p),\qquad
        \sup_{s\in\F_p}
        \left|\sum_{x\in H_0}\chi(x-s)\right|\ll\sqrt p.
\]
Changing $O(\sqrt p)$ memberships produces $H_p$ of size
$(p-1)/2$ while preserving the $O(\sqrt p)$ character-sum bound. Since, for every $a\in\F_p^\times$,
\[
        \#\{z\in\F_p:s+az^2\in H_p\}
        =|H_p|+\chi(a)\sum_{x\in H_p}\chi(x-s),
\]
we recover Lemma~\ref{lem:unif_shift}. Thus the lemma may be viewed either as a consequence of a character-sum estimate or as a discrepancy statement proved using probabilistic-combinatorial methods.

This construction also shows that the $p^{-1/2}$ scale of the local logarithmic gain in Proposition~\ref{prop:local} is optimal up to constants when the residue sets have density approximately $1/2$. Indeed, let $\mu_p$ be the uniform measure on $H_p$ and set $R_{s,a}(y)=\#\{z\in\F_p:y=s+az^2\}.$ Lemma~\ref{lem:unif_shift} gives
\[
        \E_{\mu_p}R_{s,a}\le 1+C_0p^{-1/2}
\]
uniformly in $s\in\F_p$ and $a\in\F_p^\times$. Hence, for any convex combination $Z$ of the functions $R_{s,a}$, Jensen's inequality gives
\[
        \E_{\mu_p}\log Z
        \le \log \E_{\mu_p}Z
        \ll p^{-1/2},
\]
with the left-hand side interpreted as $-\infty$ if $Z$ vanishes at some point of $H_p$.
\end{remark}

\subsection{The global construction}
We first reduce the uniformity over all integral quadratics to a
family of $O(N^3)$ normalized polynomials.
\begin{lemma}\label{lem:quadratic-normalization}
Let $N\ge3$ and let $q\in\Z[x]$ have degree two. If
$|q(\Z)\cap[N]|\ge3$, then there is an integer $t$ such that $       \widetilde q(x):=q(x+t)=ax^2+bx+c$ satisfies
\[
        1\le |a|\le N-1,\qquad
        |b|\le2(N-1),\qquad
        1\le c\le N.
\]
In particular, $\widetilde q(\Z)=q(\Z)$.
\end{lemma}
\begin{proof}
Write $q(x)=ax^2+\beta x+\gamma$, where $a\ne0$.
Choose integers $x_1<x_2<x_3$ such that the values $y_i=q(x_i)$
are distinct and lie in $[N]$, and set $h_1=x_2-x_1$ and
$h_2=x_3-x_2$. Since $q$ is quadratic,
\[
\frac{y_2-y_1}{h_1}=a(x_1+x_2)+\beta,
\qquad
\frac{y_3-y_2}{h_2}=a(x_2+x_3)+\beta.
\]
Subtracting these identities and rearranging gives
\[
        ah_1h_2
        =\frac{h_2y_1+h_1y_3}{h_1+h_2}-y_2.
\]
The fraction on the right is a convex combination of $y_1$ and $y_3$,
so $|a|h_1h_2\le N-1$. Now write
$\widetilde q(x)=q(x+x_1)=ax^2+bx+c$.
Then $c=y_1$, and $b=2ax_1+\beta$. By the first identity above,
\[
        b=\frac{y_2-y_1}{h_1}-ah_1,
        \qquad
        |b|\le\frac{N-1}{h_1}+\frac{N-1}{h_2}
             \le2(N-1). \qedhere
\]
\end{proof}

\begin{proof}[Proof of Theorem~\ref{thm:sharpness}]
Fix $0<\theta<1/2$ and set
\[
        y=\theta\log N,\qquad Q=\prod_{p\le y}p,\qquad
        S=\sum_{3\le p\le y}p^{-1/2}.
\]
The prime number theorem and partial summation give
\begin{equation}\label{eq:cubic-prime-budget}
        Q=N^{\theta+o(1)}=o(\sqrt N),\qquad
        S\asymp_\theta\frac{\sqrt{\log N}}{\log\log N}.
\end{equation}
Take $H_p$ from Lemma~\ref{lem:unif_shift} for odd $p\le y$, set
$H_2=\{0\}$, and define
\[
        B=\{n\in[N]:n\bmod p\in H_p\text{ for every }p\le y\},
        \qquad
        \varrho=\prod_{p\le y}\frac{|H_p|}{p}.
\]
Writing $k=\#\{p:p\le y\}$, we have
$\varrho\ge3^{-k}=N^{-o(1)}$. There are exactly $Q\varrho$ admissible
residue classes modulo $Q$. Thus
\begin{equation}\label{eq:cubic-crt-size}
        |B|=\left(\frac NQ+O(1)\right)Q\varrho
             =(1+o(1))N\varrho.
\end{equation}

Let
\[
        \mathscr F_N=
        \{ax^2+bx+c:a,b,c\in\Z,\ 1\le|a|\le N,
          \ |b|\le2N,\ 1\le c\le N\}.
\]
By Lemma~\ref{lem:quadratic-normalization}, every integral quadratic
whose image meets $[N]$ in at least three points has the same image
as a member of $\mathscr F_N$. Moreover,
\begin{equation}\label{eq:quadratic-family-size}
        |\mathscr F_N|=2N(4N+1)N\le10N^3.
\end{equation}

Fix $q(x)=ax^2+bx+c\in\mathscr F_N$. At an odd prime $p\nmid a$,
complete the square in $\F_p$:
\[
        q(z)=a\left(z+\frac{b}{2a}\right)^2
                   +c-\frac{b^2}{4a}.
\]
Translation of $z$ permutes $\F_p$, so Lemma~\ref{lem:unif_shift}
bounds the number of admissible parameter residues by
$|H_p|(1+C_0p^{-1/2})$. At an odd prime $p\mid a$, use the trivial
bound $p\le3|H_p|$. At $p=2$, use the bound $2=2|H_2|$,
regardless of the coefficients of $q$.
Writing $\omega(m)$ for the number of distinct prime divisors of a
positive integer $m$, with $\omega(1)=0$, the Chinese remainder
theorem therefore gives
\begin{equation}\label{eq:cubic-parameter-residues}
 \#\{z\bmod Q:q(z)\bmod p\in H_p\text{ for all }p\le y\}
 \le 2Q\varrho e^{C_0S}3^{\omega(|a|)}.
\end{equation}

Split the real preimage $q^{-1}([1,N])$ at the vertex $-b/(2a)$.
On each side it is an interval, possibly empty, of length at most
$\sqrt{N/|a|}$. Indeed, completing the square reduces each side to
taking square roots of a nonnegative interval of length at most
$(N-1)/|a|$, and
$\sqrt v-\sqrt u\le\sqrt{v-u}$ for $0\le u\le v$.
Consequently, each residue class modulo $Q$ has at most
$2(\sqrt{N/|a|}/Q+1)$ integer representatives in this preimage.
Counting parameter values can only overcount the distinct polynomial
values. It follows from estimate~\eqref{eq:cubic-parameter-residues} that
\begin{equation}\label{eq:cubic-crt-intersection}
        |B\cap q(\Z)|
        \le4\varrho e^{C_0S}3^{\omega(|a|)}
              \left(\sqrt{\frac N{|a|}}+Q\right).
\end{equation}

The dependence on $a$ in this estimate is essential. We use the
standard bound $3^{\omega(m)}\ll_\varepsilon m^\varepsilon$ for each
$\varepsilon>0$. Taking $\varepsilon=1/4$ and
$\varepsilon=(1/2-\theta)/2$, respectively, and using
$|a|\le N$ and the estimates in~\eqref{eq:cubic-prime-budget}, we obtain
\begin{equation}\label{eq:quadratic-boundary-error}
        \frac{3^{\omega(|a|)}}{\sqrt{|a|}}\ll1,\qquad
        \sup_{1\le m\le N}\frac{Q3^{\omega(m)}}{\sqrt N}
        =o_\theta(1).
\end{equation}

Set $M=\lfloor\sqrt N\rfloor$ and choose $A$ uniformly from the
$M$-element subsets of $B$; this is possible by estimate~\eqref{eq:cubic-crt-size}. For $q\in\mathscr F_N$, let
$X_q=|A\cap q(\Z)|$ and $r_q=|B\cap q(\Z)|$.
For sufficiently large $N$ in terms of $\theta$, we have
$|B|\ge N\varrho/2$, so estimates~\eqref{eq:cubic-crt-intersection}
and~\eqref{eq:quadratic-boundary-error} give
\[
        \E X_q=\frac{Mr_q}{|B|}
        \le 8e^{C_0S}\left(
             \frac{3^{\omega(|a|)}}{\sqrt{|a|}}
             +\frac{Q3^{\omega(|a|)}}{\sqrt N}\right) \leq C_1e^{C_0S}=:\mu_0,
\]
uniformly over $q\in\mathscr F_N$, for some absolute constant
$C_1>0$.

For $1\le t\le\min(r_q,M)$, the event $X_q\ge t$ implies
that $A$ contains a $t$-element subset of $B\cap q(\Z)$.
Since $A$ is chosen uniformly among the $M$-element subsets of $B$,
the union bound gives
\[
\begin{aligned}
\Prob(X_q\ge t)
\le\binom{r_q}{t}
    \frac{\binom{|B|-t}{M-t}}{\binom{|B|}{M}}
 =\binom{r_q}{t}\frac{(M)_t}{(|B|)_t} \le \frac{r_q^t}{t!}\left(\frac{M}{|B|}\right)^t
 \le \left(\frac{e r_q M}{t|B|}\right)^t
 \le \left(\frac{e\mu_0}{t}\right)^t.
\end{aligned}
\]
Here $(u)_t=u(u-1)\cdots(u-t+1)$, and we used $(M)_t/(|B|)_t\le(M/|B|)^t$ and $t!\ge(t/e)^t$. The probability is zero if $t>\min(r_q,M)$. 
Take
\[
        t=\left\lceil e^2\mu_0+10\log N\right\rceil.
\]
For each $q\in\mathscr F_N$, the tail probability is at most
$e^{-t}\le N^{-10}$. Since $|\mathscr F_N|\ll N^3$ by estimate~\eqref{eq:quadratic-family-size},
a union bound yields a choice of $A$ for which $X_q<t$ for every
$q\in\mathscr F_N$. By Lemma~\ref{lem:quadratic-normalization},
this also holds for every integral quadratic whose image meets $[N]$
in at least three points. All remaining quadratics meet $A$ in at
most two points, which is also less than $t$. We conclude that
\[
        \max\{|A\cap q(\Z)|: q\in\Z[x], \deg q=2\}
        <t
        \le\exp\left(K\frac{\sqrt{\log N}}{\log\log N}\right)
\]
for a constant $K=K(\theta)>0$ and sufficiently large $N$, by
the estimates in~\eqref{eq:cubic-prime-budget}. Finally, $A\subseteq B$ implies
$|A_p|\le(p-1)/2$ at every odd prime $p\le\theta\log N$,
and $|A_2|\le1$. This proves all the claims.
\end{proof}

\begin{remark}
The same construction applies to a logarithmic prime window $Y\le p\le2Y$, with $Y=\eta\log N$ and fixed $0<\eta<1/2$: the prime number theorem gives $\prod_{Y\le p\le2Y}p=N^{\eta+o(1)}$, and the accumulated exponent again has order $\sqrt{\log N}/\log\log N$. The leading-coefficient estimates above remain uniform, so this variant also controls all integral quadratics. Using every prime up to $\theta\log N$ instead gives the matching hypothesis used above. Theorem~\ref{thm:sharpness} is uniform over polynomials in $\Z[x]$ of degree two, without any restriction on their coefficients; it does not assert a uniform bound over arbitrary rational quadratics. The construction imposes the residue restriction only at the specified small primes, and hence does not settle sharpness for the global inverse large sieve problem in Conjecture~\ref{conj1}.
\end{remark}

\section*{Acknowledgments}
During the preparation of this paper, the authors used ChatGPT by OpenAI interactively to assist in discussing, exploring, presenting, and refining proof strategies throughout the paper. The mathematical ideas, underlying intuition, and overall direction of the work are due to the authors.

\end{document}